\documentclass[oneside,openany,a4paper]{amsart}
\usepackage[a4paper, margin=1in]{geometry}
\usepackage{latexsym, amssymb, amscd, amsthm, amsxtra, amsmath,amsthm}
\usepackage{graphics, graphicx, color}
\usepackage{ifpdf}
\usepackage[format=hang,indention=-1cm,small]{caption}
\usepackage[caption=false]{subfig}
\usepackage{multirow}
\usepackage{hyperref}
\usepackage{mathrsfs}
\usepackage[utf8]{inputenc}
\usepackage[english]{babel}
\usepackage{pst-node}
\usepackage{quiver}
\usepackage{tikz}

\newtheorem{thm}{Theorem}[section]

\newtheorem{lem}[thm]{Lemma}
\newtheorem{prop}[thm]{Proposition}
\theoremstyle{definition}

\theoremstyle{remark}

\newcounter{cases}
\newcounter{subcases}[cases]
\newenvironment{mycases}
  {%
    \setcounter{cases}{0}%
    \setcounter{subcases}{0}%
    \def\case
      {%
        \par\noindent
        \refstepcounter{cases}%
        \textbf{Case \thecases.}
      }%
    \def\subcase
      {%
        \par\noindent
        \refstepcounter{subcases}%
        \textit{Subcase \thecases-(\thesubcases):}
      }%
  }
  {%
    \par
  }
\renewcommand*\thecases{\arabic{cases}}
\renewcommand*\thesubcases{\roman{subcases}}

\newcommand{\ZZ}{\mathbb{Z}}
\newcommand{\ZP}{\mathbb{Z}_{\geq 0}}
\newcommand{\QQ}{\mathbb{Q}}
\newcommand{\ord}{\mathcal{O}}
\newcommand{\al}{\alpha}
\newcommand{\cf}{K=\QQ(\sqrt[3]{m})}
\newcommand{\nt}[1]{\frac{#1}{3}}

\newcommand{\lmat}{\mathcal{M}^d}
\newcommand{\lmatn}{\mathcal{M}_n^d}
\newcommand{\lmatord}{\mathcal{M}_{n,K}}
\newcommand{\pmatord}{\mathcal{M}_{p^n,K}}
\newcommand{\pmono}{\mathcal{M}_{p^n,K}^{\operatorname{mono}}}
\newcommand{\pord}[1]{\nu_p\left(#1\right)}
\newcommand{\ornum}{A_{K,p,n}}
\newcommand{\monnum}{B_{K,p,n}}

\title{The proportion of monogenic orders of prime power indices of the pure cubic field}
\author{Minchan Kang}
\address{Department of Mathematical Sciences, Seoul National University, Gwanak-ro 1, Gwankak-gu, Seoul, South Korea 08826}
\email{azboy@snu.ac.kr}
\author{Dohyeong Kim}
\address{Department of Mathematical Sciences and Research Institute of Mathematics, Seoul National University, Gwanak-ro 1, Gwankak-gu, Seoul, South Korea 08826}
\email{dohyeongkim@snu.ac.kr}
\begin{document}

\maketitle

\begin{abstract}
In this paper, we investigate the proportion of monogenic orders among the orders whose indices are a power of a fixed prime in a pure cubic field. We prove that the proportion is zero for a prime number that is not equal to 2 or 3. To do this, we first count the number of orders whose indices are power of a fixed prime. This is done by considering every full rank submodules of the ring of integers, and establishing the condition to be closed under multiplication. Next, we derive the index form of arbitrary orders based on the index form of the ring of integers. Then, we obtain an upper bound of the number of monogenic orders with prime power indices by applying the finiteness of the number of primitive solutions of the Thue-Mahler equation.
\end{abstract}

\section{Introduction}
Let~$K$ be a number field of degree~$d\in\ZZ$, and let~$\ord_K$ be the ring of integers. The field~$K$ is said to be \emph{monogenic} if there exists an element~$\al\in \ord_K$ such that~$\ord_K=\ZZ[\al]$, or equivalently that~$\{1,\al,\al^2,\ldots,\al^{d-1}\}$ forms an integral basis of~$K$. Note that all quadratic fields are monogenic, whereas a number field of degree at least 3 is not always monogenic. In \cite{al1,al2}, the authors demonstrated that a positive proportion of cubic and quartic number fields fail to be monogenic, even if they have no local obstructions.

Monogeniety of a given number field can be analyzed using the notion of \emph{index form}. Let~$\mathcal{B}=\{1,\al_1,\ldots,\al_{d-1}\}$ be an ordered integral basis of a given number field~$K$. This gives isomorphisms~$\ord_K/\ZZ\cong\ZZ^{d-1}$ and~$\wedge^d\ord_K\cong\ZZ$. The index form of~$K$ with respect to~$\mathcal{B}$ is the map
$$I_{\ord_K}^\mathcal{B}:\ZZ^{d-1}\cong\ord_K/\ZZ \to \wedge^d\ord_K\cong\ZZ$$
defined by~$I_{\ord_K}^\mathcal{B}(\al)=1\wedge\al\wedge\cdots\wedge\al^{d-1}$. Note that the index form is a homogeneous form of degree~${d \choose 2}$ in~$d-1$ variables. The number field~$K$ is monogenic if and only if there exists some~$\al\in\ord_K/\ZZ$ such that~$|I_{\ord_K}^\mathcal{B}(\al)|=1$. A detailed discussion of the index form is provided in \cite{IntegralBasis}.

More generally, recall that a subring of~$K$ is called an order if it is a finitely generated~$\ZZ$-submodule of~$K$ with rank~$d$. It is well known that the ring of integers is the maximal order of~$K$. The \emph{index} of an order~$\ord$ is the index~$[\ord_K:\ord]$ as a~$\ZZ$-submodule. An order~$\ord$ is said to be \emph{monogenic} if there exists an element~$\al\in\ord$ such that~$\ord=\ZZ[\al]$, or equivalently that~$\{1,\al,\al^2,\ldots,\al^{d-1}\}$ forms a~$\ZZ$-basis of~$\ord$. For instance, any order of a quadratic field is monogenic, and for any given number field, there are infinitely many monogenic orders. The definition of the index form extends to an arbitrary order. This is discussed in \S 3.

In this paper, we investigate the orders of a pure cubic field -- $\QQ(\sqrt[3]{m})$, where~$m\in\ZZ$ is a cube free integer -- such that~$m^2\not\equiv1\pmod{9}$. Throughout this paper, a fixed prime number which is not~$2$ or~$3$ is denoted by~$p$. We focus on orders whose indices are powers of~$p$.  We show that the proportion of monogenic orders among the orders with $p$-power indices is zero. Precisely, we prove:

\begin{thm}
Let~$K=\QQ(\sqrt[3]{m})$ be the pure cubic field for a cube-free integer~$m\in\ZZ$ satisfying~$m^2\not\equiv1\pmod{9}$, and let~$p\not=2,3$ be a prime number. Let~$\ornum$ be the number of orders of~$K$ of index~$p^i$ for some~$i\in\{0,1,2,\ldots,n\}$, and~$\monnum$ the number of monogenic orders of index~$p^i$ for some~$i\in\{0,1,2,\ldots,n\}$. Then we have
$$\left[\nt{n}\right]\leq\log_p\ornum\leq n\quad\text{ and }\quad\monnum=O(n).$$
In particular,
$$\lim_{n\to\infty}\frac{\monnum}{\ornum}=0.$$
\end{thm}

To prove Theorem 1.1, we obtain a closed formula of~$\ornum$ and an upper bound of~$\monnum$. For~$\ornum$, we obtain the exact condition for a~$\ZZ$-submodule of~$\ord_K$ to be an order. We then get an upper bound of~$\monnum$ by analyzing the index form of the order and the aforementioned condition we found.

To get~$\ornum$, we first consider all~$\ZZ$-submodules of index~$p^n$ whose~$\ZZ$-basis contains 1. These submodules correspond to the matrices in the Hermite normal form. We then investigate the condition for a matrix to ensure that the corresponding~$\ZZ$-submodule becomes a subring of~$K$, meaning it is closed under multiplication. We use this condition to evaluate~$\ornum$. In \cite{MR3490755}, the authors provided an asymptotic formula for the number of orders of the number field with respect to the index. Our formula refines the asymptotic formula by establishing the condition of the corresponding Hermite normal form, moreover it plays a crucial role in our discussion of~$\monnum$.

To derive an upper bound of~$\monnum$, we express the index form of the order in terms of the index form of the ring of integers. Then, we reduce our problem to solving a specific Thue-Mahler equation. Then the finiteness of the number of primitive solutions of a Thue-Mahler equation provides an upper bound of~$\monnum$.

This paper is organized as follows. In \S 2, we establish a closed formula for~$\ornum$. In \S 3, we relate the index form of the order to the index form of the ring of integers. Finally, in \S 4, we obtain the upper bound of~$\monnum$, and finish the proof of Theorem 1.1.
 
\subsection*{Acknowledgement}
{Dohyeong Kim is partially supported by the NRF grant funded by the Korean government (MSIT) (No.\,2020R1C1C1A01006819) and the Samsung Science \& Technology Foundation (No.\,SSTF-BA2001-01). This work was supported by the SNU Student-Directed Education Undergraduate Research Program through the Faculty of Liberal Education, Seoul National University (2023).}
\section{The number of orders of prime power indices}

In this section, we establish a closed formula for~$\ornum$. Let
$$\lmat=\left\{\left.\begin{pmatrix}  a_{11} & 0 & \cdots &0  \\ a_{21} & a_{22} &\cdots&0 \\ \vdots &\vdots &\ddots &\vdots\\ a_{d1} &a_{d2} &\cdots &a_{dd}\end{pmatrix}\in \text{Mat}_{d\times d}(\ZZ) \right| a_{ij}\geq0, \,a_{ij}<a_{jj} \text{ for } i>j\right\}$$
be the set of matrices in Hermite normal forms. The following correspondence is well-known.
\begin{prop}
Let~$A=\bigoplus_{i=1}^d\ZZ \al_i$ be a free~$\ZZ$-module of rank~$d$. The submodules of~$A$ of rank~$d$ are in a one-to-one correspondence with elements of~$\lmat$. Precisely, the correspondence is given by~$M\mapsto \bigoplus_{i=1}^d\ZZ v_i$, where
\begin{equation}\label{basis}
    \begin{pmatrix}  v_1 \\ v_2 \\ \vdots \\ v_d \end{pmatrix}=M\begin{pmatrix}  \al_1 \\ \al_2 \\ \vdots \\ \al_d \end{pmatrix}.
\end{equation}
\end{prop}

Let~$\lmatn$ be the subset of~$\lmat$ whose elements have determinant~$n$. For a number field~$K$ of degree~$d$, fix an ordered integral basis~$\{1,\al_1,\ldots,\al_{d-1}\}$. For 
$M\in\lmat$, denote by~$\ord_M$ the~$\ZZ$-submodule of~$\ord_K$ corresponding to~$M$. Then~$\ord_M$ has a preferred ordered~$\ZZ$-basis $\{v_1,v_2,\ldots,v_d\}$ as in (\ref{basis}). Let 
$$\lmatord:=\{M\in\lmatn\mid \ord_M\text{ is an order}\}.$$
By Proposition 2.1, the set of orders of index~$n$ are in a bijection with~$\lmatord$. 
\begin{prop}
    Let $\ord_M$ be a $\ZZ$-submodule of $\ord_K$ corrseponding to $M\in\lmat$, and $\{v_1,v_2,\ldots,v_d\}$ the preferred ordered basis as in (\ref{basis}). If $\ord_M$ is an order, then $v_1=1$.
\end{prop}
\begin{proof}
    Any order contains $1\in\ZZ$. If $v_1>1$, the linear combination of $v_1,v_2,\ldots,v_d$ cannot be 1.
\end{proof}
In particular, for
$$M=\begin{pmatrix}  a_{11} & 0 & \cdots &0  \\ a_{21} & a_{22} &\cdots&0 \\ \vdots &\vdots &\ddots &\vdots\\ a_{d1} &a_{d2} &\cdots &a_{dd}\end{pmatrix}\in \lmatord,$$
we always have~$a_{11}=1$ and~$a_{i1}=0$ for~$i>1$. Let~$\iota:\mathcal{M}_{n}^{d-1}\hookrightarrow\lmatn$ be a map
$$M\to \begin{pmatrix} 1&0\\0&M\end{pmatrix}.$$
Proposition 2.2 says that~$\lmatord\subset \iota(\mathcal{M}_{n}^{d-1})$. Thus, from now on, we consider elements of~$M\in\mathcal{M}_{n}^{d-1}$ and investigate the condition that~$\iota(M)\in\lmatord$.

In the rest of this section, let~$K=\QQ(\sqrt[3]{m})$ with~$m^2\not\equiv1\pmod{9}$.

\begin{thm}[{\cite[Theorem 7.3.2]{alaca_williams_2003}}]
Let~$m\in\mathbb{Z}$ be a cube-free integer. Set~$m=hk^2$ for square-free integers~$h, k\in\mathbb{Z}$ with~$(h,k)=1$. Let~$\theta=\sqrt[3]{m}$ and~$K=\mathbb{Q}(\theta)$. Then an integral basis of~$K$ is given by
$$\left\{1,\theta,\frac{\theta^2}{k}\right\} \quad{}\text{ if}\quad{}m^2\not\equiv1\pmod{9},$$
$$\left\{1,\theta,\frac{k^2\pm k^2\theta+\theta^2}{3k}\right\} \quad{}\text{ if}\quad{}m\equiv\pm1\pmod{9}.$$
\end{thm}

Let~$X=\theta$ and~$Y=\theta^2/k$, and take~$M=\begin{pmatrix} a & 0 \\ b & c \end{pmatrix} \in \mathcal{M}_n^2$, so that~$\ord_{\iota(M)}$ has~$\{1,aX,bX+cY\}$ as a~$\ZZ$-basis. Then~$\iota(M)$ belongs to $\lmatord$ if and only if~$\ord_{\iota(M)}$ is closed under multiplication. This holds if and only if~$(aX)^2,aX(bX+cY)$, and $(cY)^2$ are contained in~$\ord_{\iota(M)}$, which means they are~$\ZZ$-linear combinations of~1, $aX$, and $bX+cY$, or equivalently are~$\ZZ$-linear combinations of~$aX$ and~$bX+cY$ in~$\ord_K/\ZZ$. In~$\mathcal{O}_K/\mathbb{Z}$, we have~$X^2=kY$,~$XY=0$, and~$Y^2=hX$. Thus, for $M\in\mathcal{M}_n^2$,~$\iota(M)\in\lmatord$ if and only if 
\begin{equation}\label{eu_eqn}
\begin{split}
(aX)^2&=a^2X^2=ka^2Y\\ 
aX(bX+cY)&=abX^2=kabY \\ 
(bX+cY)^2&=b^2X^2+c^2Y^2=kb^2Y+hc^2X
\end{split}
\end{equation}
are~$\mathbb{Z}$-linear sums of~$aX$ and~$bX+cY$.

\begin{prop}
Let~$r_p$ be the number of roots of~$X^3-m=0$ in~$\ZZ/p\ZZ$ counted without multiplicity. Then
$$
r_p=
\begin{cases}
0\quad{}\text{ if}\quad{}p\equiv1\pmod{3}\quad{}\text{and}\quad{}m^{\frac{p-1}{3}}\not\equiv 1\pmod{p},\\
1\quad{}\text{ if}\quad{}p\equiv2\pmod{3}\quad{}\text{or}\quad{}p\mid m,\\
3\quad{}\text{ if}\quad{}p\equiv1\pmod{3}\quad{}\text{and}\quad{}m^{\frac{p-1}{3}}\equiv 1\pmod{p}.\\
\end{cases}
$$
\end{prop}
\begin{proof}
    This is a well-known result from elementary number theory, which comes from the order argument in the cyclic group~$\mathbb{F}_p^\times$.
\end{proof}
Denote by~$\pord{\cdot}$ the~$p$-adic valuation in~$\QQ$, normalized as~$\pord{p}=1$. Now we are ready to give a closed formula for $|\pmatord|$.

\begin{thm}
Let~$m\in\ZZ$ be a cube-free integer with~$m^2\not\equiv1\pmod{9}$. Set~$m=hk^2$ for square free integers~$h, k\in\ZZ$ with~$(h,k)=1$. Let~$\cf$, and~$r_p$ as in Proposition 2.3. 
Then
$$\left|\pmatord\right|=
\begin{cases}
\frac{p^{\left[\frac{n}{3}\right]+1}-1}{p-1}\quad{}\text{ if}\quad{}p\mid m,\\
\\
\frac{p^{\left[\frac{2n}{3}\right]-\left\lceil\frac{n}{3}\right\rceil+1}-1}{p-1}+\frac{p^{\left\lceil\frac{n}{3}\right\rceil}-1}{p-1}r_p\quad{}\text{ if}\quad{}p\nmid m.\\
\end{cases}$$
\end{thm}
\begin{proof}
Fix~$n\in\ZP$, and consider $M\in\mathcal{M}_{p^n}^2$. Then $M$ is of the form
$\begin{pmatrix} p^i & 0 \\ \beta & p^j \end{pmatrix}$ such that~$i+j=n$ and~$\beta<p^i$. We characterize the subset $\iota^{-1}(\pmatord)\subset\mathcal{M}_{p^n}^2$ in terms of~$(i,j,\beta)$. By (\ref{eu_eqn}), $\iota(M)$ belongs to $\pmatord$ if and only if the coefficients of~$p^iX$ and~$\beta X+p^jY$ in
\begin{equation*}
\begin{split}
kp^{2i}Y&=kp^{2i-j}(\beta X +p^jY)-kp^{i-j}\beta(p^iX)\\ 
kp^{i}\beta Y&=kp^{i-j}\beta(\beta X +p^jY)-kp^{-j}\beta^2(p^iX) \\ 
k\beta^2Y+hp^{2j}X&=kp^{-j}\beta^2(\beta X +p^jY)+(hp^{2j-i}-kp^{-i-j}\beta^3)(p^iX)
\end{split}
\end{equation*}
are integers. These coefficients are integers if and only if four numbers
$$kp^{2i-j},kp^{i-j}\beta,kp^{-j}\beta^2,hp^{2j-i}-kp^{-i-j}\beta^3$$
have nonnegative $p$-adic valuations. If $\beta=0$, it is enough to consider $kp^{2i-j}$ and $hp^{2j-i}$. If $\beta\not=0$, then $\pord{\beta}<i$ as $\beta<p^i$. Hence $\pord{kp^{2i-j}}\geq\pord{kp^{i-j}\beta}\geq \pord{kp^{-j}\beta^2}$, so it is enough to consider $kp^{2i-j}$ and $hp^{2j-i}-kp^{-i-j}\beta^3$. Thus~$M\in\iota^{-1}(\pmatord)$ if and only if one of the following two conditions is satisfied:
\begin{itemize}
    \item[(i)]~$\beta=0$,~$2i-j+\pord{k}\geq0$ and~$2j-i+\pord{h}\geq 0$,
    \item[(ii)]~$\beta\not=0$,~$\pord{k\beta^2}-j\geq0$ and~$\pord{hp^{2j-i}-kp^{-i-j}\beta^3}\geq0.$
\end{itemize}
Since $i+j=n$, the condition (i) is equivalent to
$$\beta=0,\nt{n-\pord{k}}\leq i \leq \nt{2n+\pord{h}}.$$
Now we use the above characterization to count the number of triples $(i,j,\beta)$ such that the corresponding matrix is contained in $\iota^{-1}(\pmatord)$.
\begin{mycases}
\case
Suppose~$p\mid m$. Then we have the following observation.
\begin{lem}
For $M\in\mathcal{M}_{p^n}^2$, corresponding $(i,j,\beta)$ satisfies condition (ii) if and only if they satisfy all of
 $$\beta\not=0,\,i\leq\nt{2n+\pord{h}},\text{ and }\pord{\beta}\geq \nt{n-\pord{k}}.$$
\end{lem}
\begin{proof}
    This comes from the nonarchimedean property of $\nu_p$. Note that one of $\pord{h}$ and $\pord{k}$ is 1, and the other is 0 since $m$ is cube-free. Then
    $$\pord{hp^{2j-i}}-\pord{kp^{-i-j}\beta^3}=\pord{h}-\pord{k}+3j-3\pord{\beta}\not\equiv 0\pmod{3},$$
    so 
    $$\pord{hp^{2j-i}-kp^{-i-j}\beta^3}\geq 0 \Longleftrightarrow\pord{hp^{2j-i}}\geq0\text{ and }\pord{kp^{-i-j}\beta^3}\geq 0.$$
    Moreover, if $\pord{kp^{-i-j}\beta^3}\geq 0$, then $\pord{k\beta^2}\geq j$ follows automatically since $\pord{\beta}<i$. Finally, since $i+j=n$, we have
    $$\pord{hp^{2j-i}}\geq0\text{ and }\pord{kp^{-i-j}\beta^3}\geq 0\Longleftrightarrow i\leq\nt{2n+\pord{h}}\text{ and }\pord{\beta}\geq\nt{n-\pord{k}}.$$
    Thus our claim is proved.
\end{proof}

By Lemma 2.6, we conclude that~$\iota(M)\in\pmatord$ if and only if
\begin{equation}\label{pm cond}
    \beta=0\text{ or }\pord{\beta}\geq\frac{n-\pord{k}}{3},\quad\frac{n-\pord{k}}{3}\leq i\leq\frac{2n+\pord{h}}{3}.
\end{equation}
Since one of~$\pord{k}$ and~$\pord{h}$ is 0 and the other is 1, we have
\begin{equation}\label{ab}
    \left[\frac{2n+\pord{h}}{3}\right]-\left\lceil\frac{n-\pord{k}}{3}\right\rceil=\left[\frac{n}{3}\right].
\end{equation}
Let 
$a=\left\lceil(n-\pord{k})/3\right\rceil$ and~$b=\left[(2n+\pord{h})/3\right]$. For fixed~$i$, the number of~$0\leq\beta<p^i$ satisfying (\ref{pm cond}) is~$p^{i-a}$. Thus by \eqref{ab}, the number of the triple~$(i,j,\beta)$ satisfying (\ref{pm cond}) is 
\begin{equation*}
\sum_{i=a}^{b}p^{i-a}
=\frac{p^{\left[\frac{n}{3}\right]+1}-1}{p-1}.
\end{equation*}
\case
Suppose~$p\nmid m$. Note that~$(i,j,\beta)$ satisfying the same condition (\ref{pm cond}) gives rise to an order in this case. However, we have additional possibilities in (ii) which does not satisfy (\ref{pm cond}). Indeed, even if~$\pord{hp^{2j-i}}<0\text{ or }\pord{kp^{-i-j}\beta^3}<0$, if
$$\pord{hp^{2j-i}}=\pord{kp^{-i-j}\beta^3},$$
which is equivalent to~$\pord{\beta}=j$, the condition (ii) is satisfied when
\begin{equation}\label{pnm add}
    \pord{h-k\al^3}=\pord{m-(k\al)^3}\geq i-2j,
\end{equation}
where $\al=p^{-\pord{\beta}}\beta$. Since~$\pord{hp^{2j-i}}<0$, we require $j<n/3$. Note that 
$$\pord{k\beta^2}-j=j\geq0$$
is automatically satisfied. As~$\al<p^{i-\pord{\beta}}=p^{i-j}$, the number of~$\al$ satisfying (\ref{pnm add}) is~$p^jr_p$ by applying Hensel's lemma in the equation $x^3-m=0$. Thus, the number of the triple~$(i,j,\beta)$ satisfying (i) or (ii) is
\begin{equation*}
\sum_{i=\left\lceil\frac{n}{3}\right\rceil}^{\left[\frac{2n}{3}\right]}p^{i-\left\lceil\frac{n}{3}\right\rceil}+\sum_{j=0}^{\left\lceil\nt{n}\right\rceil-1}p^jr_p
=\frac{p^{\left[\frac{2n}{3}\right]-\left\lceil\frac{n}{3}\right\rceil+1}-1}{p-1}+\frac{p^{\left\lceil\frac{n}{3}\right\rceil}-1}{p-1}r_p.\\
\end{equation*}
\end{mycases}
\end{proof}

\section{Index form of the order}
To analyze the monogeneity of the order, we extend the notion of \emph{index form} to orders. For an order~$\ord$ of the number field~$K$ of degree~$d$ with a~$\ZZ$-basis~$\mathcal{B}$, let
$$I_{\ord}^\mathcal{B}:\ord/\ZZ\cong\ZZ^{d-1} \to \ZZ\cong\wedge^d\ord$$
be the map given by~$\al \mapsto1\wedge\al\wedge\cdots\wedge\al^{d-1}$. Then~$I_\ord^\mathcal{B}$ is a homogeneous form of degree~${d \choose 2}$ in~$d-1$ variables. Thus~$\ord$ is monogenic if and only if the equation
$$|I_{\ord}^\mathcal{B}(x_1,x_2,\ldots,x_{d-1})|=1,\text{ }(x_1,x_2,\ldots,x_{d-1})\in\ZZ^{d-1}$$
has a solution.

The index form depends on the choice of a~$\ZZ$-basis of~$\ord$. In \S 2, we discussed that every order has a preferred orderded $\ZZ$-basis. The index form with respect to that basis is related to the index form of the ring of integers.
\begin{lem}
Let~$K$ be a number field of degree~$d$, and let~$I_{\ord_K}^\mathcal{B}$ be the index form of~$\ord_K$ with respect to the integral basis~$\mathcal{B}=\{1,\al_1,\ldots,\al_{d-1}\}$. Let~$\ord=\ord_M \subset \ord_K$ be an order of~$K$ corresponding to~$M\in\lmatord$, and~$\mathcal{B}_M$ a preferred ordered~$\ZZ$-basis of $\ord$ with respect to~$M$. Then the index form of~$\ord$ with respect to $\mathcal{B}_M$ is 
$$I_{\ord}^{\mathcal{B}_M}(x_1,x_2,\ldots,x_{d-1})=\frac{1}{n}I_{\ord_K}^\mathcal{B}((x_1,x_2,\ldots,x_{d-1})M).$$
\end{lem}
\begin{proof}
We have a commutative diagram 
\[\begin{tikzcd}
	{\mathcal{O}/\mathbb{Z}} & {\wedge^{d}\mathcal{O}} \\
	{\mathcal{O}_K/\mathbb{Z}} & {\wedge^{d}\mathcal{O}_K}
	\arrow["{I_{\mathcal{O}}^{\mathcal{B}_M}}", from=1-1, to=1-2]
	\arrow[ "\phi" ',from=1-1, to=2-1]
	\arrow[ "\psi", from=1-2, to=2-2]
	\arrow["{I_{\mathcal{O}_K}^\mathcal{B}}"', from=2-1, to=2-2]
\end{tikzcd}\]
where~$\phi, \psi$ are induced from the inclusion. For~$(x_1,x_2,\ldots,x_{d-1})\in \ord/\ZZ$, we have 
$$\phi(x_1,x_2,\ldots,x_{d-1})=(x_1,x_2,\ldots,x_{d-1})M.$$
Since~$\psi$ is the multiplication-by-$n$ map, our claim is proved.
\end{proof}
For a pure cubic field~$\cf$ with~$m=hk^2$ and~$m^2\not\equiv1\pmod{9}$, the index form of~$\ord_K$ is 
$$I_{\ord_K}^\mathcal{B}(x,y)=kx^3-hy^3$$
with respect to the integral basis~$\mathcal{B}=\left\{1,\theta,\frac{\theta^2}{k}\right\}$. By Lemma 3.1, for~$M=\begin{pmatrix} p^i & 0 \\ \beta & p^j \end{pmatrix}\in \iota^{-1}(\mathcal{M}_{p^n,K})$, $\ord_M$ is monogenic if and only if 
\begin{equation}\label{eqn:tm1}
    |k(p^ix+\beta y)^3-h(p^j y)^3|=p^{i+j}.
\end{equation}
has an integer solution $(x,y)\in\ZZ^2$.

\section{The upper bound of the number of monogenic orders}

Let
$$\pmono:=\{M\in\pmatord| \ord_M\text{ is monogenic}\}.$$
In \S 3, we stated that for~$M=\begin{pmatrix} p^i & 0 \\ \beta & p^j \end{pmatrix}\in \iota^{-1}(\mathcal{M}_{p^n,K})$, $\iota(M)\in\pmono$ if and only if (\ref{eqn:tm1}) has the integer solution~$(x,y)\in\ZZ^2$. 
To bound~$|\pmono|$, we use the finiteness result of the Thue Mahler equation.

\begin{thm}[Mahler, \cite{Mahler1933}]
Let~$F(U,V)\in\ZZ[U,V]$ be a separable binary form of~$\deg F \geq 3$, and let~$\{p_1,p_2,\ldots,p_s\}$ be prime numbers of~$\ZZ$. Fix~$c\in\ZZ$ such that~$(c,p_1p_2\cdots p_s)=1$. Consider the equation
\begin{equation}\label{eqn:tm2}
    F(U,V)=cp_1^{z_1}p_2^{z_2}\cdots p_s^{z_s}
\end{equation}
with unknowns~$(U,V,z_1,z_2,\ldots,z_s)\in\ZZ^{s+2}$. This equation has finite number of primitive solutions, where a solution $(U,V,N)$ is called primitive if~$(U,V)=1$.
\end{thm}

The equation (\ref{eqn:tm2}) is called the \emph{Thue-Mahler equation}. Theorem 4.1 says that there are only finitely many triples of~$(U,V,N)\in\ZZ^3$ which satisfies
\begin{equation}\label{eqn:tm3}
    F(U,V)=kU^3-hV^3=\pm p^N
\end{equation}
and~$(U,V)=1$. 

\begin{lem}
The primitive solutions of (\ref{eqn:tm3}) have the form of~$(p^{a}U',p^{b}V',N)$, where~$(p,U'V')=1$, and their~$(a,b,N)$ fall into the following four cases:
\begin{enumerate}
\item[(i)] $(0,0,0)$,
\item[(ii)] $(a,0,\pord{h})\quad\text{ for }a>0$,
\item[(iii)] $(0,b,\pord{k})\quad\text{ for }b>0$,
\item[(iv)] $(0,0,N)\quad\text{ for }N>0\text{ if }p\nmid m$.
\end{enumerate}

\end{lem}
\begin{proof}
    Substituting $U=p^aU'$ and $V=p^bV'$ in (\ref{eqn:tm3}), we have
   $$kp^{3a}U'^3-hp^{3b}V'^3=\pm p^N.$$
    The primitivity implies that at least one of~$a,b$ must be zero. If $a=b=0$ and $p\mid m$, then 
    $$N=\min\{\pord{h},\pord{k}\}=0$$
    by the nonarchimedean property.
    If $a>0$, then $\pord{kp^{3a}U'^3}\geq 3a>1\geq\pord{hp^{3b}V'^3}$, so
    $$N=\min\{\pord{kp^{3a}U'^3},\pord{hp^{3b}V'^3}\}=\pord{h}.$$
    Finally, the case when $b>0$ is similar to the case $a>0$.
\end{proof}

Let~$g$ be the number of primitive solutions of (\ref{eqn:tm3}). If
$$(p^{a_1}U_1,p^{b_1}V_1,N_1), (p^{a_2}U_2,p^{b_2}V_2,N_2),\ldots,(p^{a_g}U_g,p^{b_g}V_g,N_g)$$
are primitive solutions of (\ref{eqn:tm3}), then every solution of (\ref{eqn:tm3}) has the form of
$$(p^{a_t+e}U_t,p^{b_t+e}V_t,N_t+3e)$$
for some~$t \in \{1,2,\ldots,g\}$ and~$e\in\ZZ_{\geq 0}$. Thus for~$M=\begin{pmatrix} p^i & 0 \\ \beta & p^j \end{pmatrix}\in \iota^{-1}(\pmatord)$,~$\iota(M)\in\pmono$ if and only if there exists~$t \in \{1,2,\ldots,g\}$ and~$e\in\ZZ_{\geq 0}$ such that
\begin{equation*}
\begin{split}
p^i x+\beta y &=p^{a_t+e}U_t\\ 
p^j y&= p^{b_t+e}V_t\\ 
n&=N_t+3e
\end{split}
\end{equation*}
has an integer solution~$(x,y)\in\ZZ^2$. Since the equation is linear, this holds if and only if
\begin{equation}\label{inteq}
\begin{split}
y&=p^{b_t+e-j}V_t\\
x&=p^{a_t+e-i}U_t-p^{b_t+e-j-i}\beta V_t
\end{split}
\end{equation}
are integers, and~$n=N_t+3e$ for some~$t \in \{1,2,\ldots,g\}$ and~$e\in\ZZ_{\geq 0}$.

\begin{lem}
Let~$(p^{a}U,p^{b}V,N)$ be a primitive solution of the equation (\ref{eqn:tm3}). Let~$n>N$ be an integer, and assume $3|n-N$. Let~$e\geq 0$ be the integer such that~$n=N+3e$. There are at most 2 elements in~$\pmatord$ such that the corresponding numbers
\begin{equation}\label{lemeq}
\begin{split}
y&=p^{b+e-j}V\\
x&=p^{a+e-i}U-p^{b+e-j-i}\beta V
\end{split}
\end{equation}
are integers.
\end{lem}
\begin{proof}
In (\ref{lemeq}),~$x$ and~$y$ are integers if and only if
\begin{itemize}
\item $b+e-j\geq 0$, and\\
\item $\pord{p^{a+e-i}U-p^{b+e-j-i}\beta V}\geq 0.$
\end{itemize}
We analyze the condition in two cases,~$p\mid m$ and~$p\nmid m$, appeared in the proof of Theorem 2.5, which are further divided into subcases desired in Lemma 4.2. Throughout the proof, let~$\al=p^{-\pord{\beta}}\beta$.

\begin{mycases}
\case
Suppose~$p\mid m$. Recall that
$$\iota^{-1}(\pmatord)=\left\{\left.\begin{pmatrix} p^i & 0 \\ \beta & p^j \end{pmatrix}\in\mathcal{M}_{p^n}^2\right|(i,j,\beta)\text{ satisfies } (\ref{pm cond})\right\}.$$
\subcase
If~$(a,b,N)=(0,0,0)$, we have~$e=n/3$. Then the condition is
\begin{equation}\label{N=0}
    j\leq \nt{n}\text{ and } \pord{p^{\nt{n}-i}U-p^{-\nt{2n}}\beta V}\geq 0.
\end{equation}
The first condition and (\ref{pm cond}) implies that the only possible~$i,j$ is
$$i=\nt{2n},\, j=\nt{n}.$$
Note that this implies that $\beta=0$ is impossible. Then the second condition is
$$\pord{\beta}=\nt{n} \text{ and } \pord{U-\al V}\geq \nt{n}.$$
Since~$\al<p^{i-n/3}=p^{n/3}$, at most one~$\al$ satisfies the condition.

\subcase
If~$(a,b,N)=(a,0,\pord{h})$ for~$a>0$, we have~$e=(n-\pord{h})/3$. The first condition~$j\leq e$ and (\ref{pm cond}) implies that the only possible~$i,j$ is
$$i=\nt{2n+\pord{h}},\, j=\nt{n-\pord{h}}.$$
If~$\beta\not=0$, the second condition is
$$\pord{p^{a+e-i}U-p^{-i}\beta V}\geq 0\Longleftrightarrow a=\pord{\beta}-e\text{ and }\pord{U-\al V}\geq i-\pord{\beta},$$
since~$\pord{\beta}<i$. As~$a$ determines~$\pord{\beta}$, and~$\al<p^{i-\pord{\beta}}$, there is at most one~$\al$ satisfying the condition. As~$\beta=0$ is possible, there are at most two triples of~$(i,j,\beta)$ which satisfy the condition.
\subcase
If~$(a,b,N)=(0,b,\pord{k})$ for~$b>0$, we have~$e=(n-\pord{k})/3$. First suppose~$\beta=0$. The second condition then gives~$i\leq e$, thus the only possible~$(i,j,\beta)$ is~
$$\left(\nt{n-\pord{k}},\nt{2n+\pord{k}},0\right)$$
by (\ref{pm cond}). Now suppose~$\beta\not=0$. Since~$i>(n-\pord{k})/3=e$ 
 by (\ref{pm cond}), the second condition is
$$\pord{p^{e-i}U-p^{b+e-j-i}\beta V}\geq 0\Longleftrightarrow b=j-\pord{\beta}\text{ and } \pord{U-\al V}\geq i-e.$$
Substituting~$b=j-\pord{\beta}$ in the first condition, we have~$\pord{\beta}\leq e$, so~$\pord{\beta}=e$ by (\ref{pm cond}). Then~$i$ and~$j$ is determined by~$b$, and since~$\al<p^{i-\pord{\beta}}=p^{i-e}$, there is at most one~$\al$ satisfying the condition. Thus for this subcase, there are at most two triples of~$(i,j,\beta)$ which satisfy the condition.
\case
Suppose~$p\nmid m$. Recall that
\begin{equation}\label{mcond}
\pmatord=\left\{\left.\begin{pmatrix} p^i & 0 \\ \beta & p^j \end{pmatrix}\in\mathcal{M}_{p^n}^2\right|(i,j,\beta)\text{ satisfies } (\ref{pm cond})\text{ or }j<\nt{n}, \pord{\beta}=j\text{ and }  \al \text{ satisfies }(\ref{pnm add})\right\}.
\end{equation}
\subcase
If~$(a,b,N)=(0,0,0)$, we have~$e=n/3$ and the condition is same with (\ref{N=0}). If~$j=n/3$, the same discussion with the subcase 1-(i) gives that there is at most one~$\beta$ such that~$(2n/3,n/3,\beta)$ satisfies the condition. If~$j<n/3$, we have~$\pord{\beta}=j$ by (\ref{mcond}), and
$$\pord{p^{\nt{n}-i}U-p^{-\nt{2n}}\beta V}\geq 0\Longleftrightarrow \nt{n}-i=-\nt{2n}+\pord{\beta}\text{ and } \pord{U-\al V}\geq i-\nt{n}.$$
Also, note that
\begin{equation}\label{compare}
\pord{kU^3-hV^3}=\pord{(kU)^3-mV^3}=\pord{k^3(U^3-(\al V)^3)+(k\al^3-m)V^3}.
\end{equation}
As~$(U,V,0)$ is a primitive solution of (\ref{eqn:tm3}), the left part of (\ref{compare}) is 0, while the right part is positive since~$\pord{U-\al V}\geq i-n/3$,~$i>2n/3$, and~$\al$ satisfies (\ref{pnm add}). Thus this is impossible, and there is at most one triple of~$(i,j,\beta)$ satisfying the condition in this subcase.
\subcase
If~$(a,b,N)=(a,0,0)$ for~$a>0$, we have~$e=n/3$. The first condition is~$j\leq n/3$. Suppose~$j=n/3$. Then the second condition is
$$\pord{p^{a-\nt{n}}U-p^{-\nt{2n}}\beta V}\geq 0.$$
If~$\beta\not=0$, this gives 
$$\pord{\beta}=a+\nt{n}\text{ and }\pord{U-\al V}\geq \nt{n}-a.$$
Since~$\al<p^{i-\pord{\beta}}=p^{n/3-a}$, there is at most one~$\al$ satisfying the condition. As~$\beta=0$ is also possible, there are at most two possible cases when~$j=n/3$. Now suppose~$j<n/3$. Then~$\pord{\beta}=j<n/3$ by (\ref{mcond}), so we have
$$\pord{p^{a+\nt{n}-i}U-p^{-\nt{2n}}\beta V}\geq 0\Longleftrightarrow a=\pord{\beta}+i-n\text{ and }\pord{U-\al V}\geq \nt{2n}-j.$$
However, the first part gives~$a=j+i-n=0$, and this contradicts our assumption of the subcase. Thus, there are at most two triples of~$(i,j,\beta)$ satisfying the condition in this subcase.
\subcase
If~$(a,b,N)=(0,b,0)$ for~$b>0$, we have~$e=n/3$. If~$i=n/3$, the only possible~$\beta$ is 0 by (\ref{pm cond}). If~$i>n/3$, the condition is
$$j\leq b+\nt{n},\,\pord{\beta}=j-b\text{ and } \pord{U-\al V}\geq i-\nt{n}.$$
The first and the second part with (\ref{mcond}) and the fact that~$b>0$ gives~$\pord{\beta}=n/3$. Then~$b$ determines~$j$, and as~$\al<p^{i-\pord{\beta}}=p^{i-n/3}$, there is at most one~$\al$ satisfying the condition. Thus, there are at most two triples of~$(i,j,\beta)$ satisfying the condition in this subcase.
\subcase
If~$(a,b,N)=(0,0,N)$ for~$N>0$, the condition is~$$j\leq e\text{ and } \pord{p^{e-i}U-p^{e-n}\beta V}\geq 0.$$
Note that~$e=(n-N)/3<n/3$. The first part and (\ref{mcond}) then implies that~$\pord{\beta}=j<n/3$. In (\ref{compare}), the left part is~$N$, and the condition gives~$\pord{U-\al V}\geq n-j-e\geq n-2e>N$. Hence~$\pord{k\al^3-m}=N$, and since~$\al$ satisfies (\ref{pnm add}), we get~$\pord{\beta}=j=e$. Then since~$\al<p^{i-\pord{\beta}}=p^{i-e}$, there is at most one~$\al$ satisfying the condition.
\end{mycases}
In sum, considering all the cases, there are at most two elements of~$\pmatord$ whose corresponding~$x,y$ in (\ref{lemeq}) are integers.
\end{proof}

By Lemma 4.3, we are ready to obtain the upper bound of the number of elements in~$\pmono$. Combining with Theorem 2.5, we complete the proof of Theorem 1.1.
\begin{thm}
For a sufficiently large integer~$n$ and a prime number~$p\not=2,3$,~$|\pmono|$ is bounded above by~$2g$, where~$g$ is the number of primitive solutions of the Thue-Mahler equation~$$F(U,V)=kU^3-hV^3=\pm p^N.$$
\end{thm}

\begin{proof}
Let~$n$ be sufficiently large so that~$n>N_t$ for all~$1\leq t\leq g$. For each~$t$, the number of elements in~$\pmatord$ such that the corresponding triples~$(i,j,\beta)$ making the solution~$x,y$ of (\ref{inteq}) be integers is at most 2 by Lemma 4.3. Thus there are at most~$2g$ elements of~$\pmatord$ which is contained in~$\pmono$.
\end{proof}

\begin{proof}[Proof of Theorem 1.1]
Since $\ornum=\sum_{k=0}^n|\mathcal{M}_{p^k,K}|$, Theorem 2.5 implies that
$$p^{\left[\nt{n}\right]}\leq \ornum\leq p^n.$$
Moreover, Theorem 4.4 implies that $\monnum=O(n)$ since $\monnum=\sum_{k=0}^n|\mathcal{M}_{p^k,K}^{\operatorname{mono}}|$.
\end{proof}

\bibliographystyle{plain}
\bibliography{MinchanKangorder}

\end{document}